\documentclass[a4paper,11pt,twoside]{article}

\usepackage[authoryear]{natbib}
\usepackage[dvipdfmx]{graphicx}
\usepackage{amsmath}
\usepackage{amsthm}
\usepackage{amssymb}
\usepackage{color}
\def\qed{\hfill $\Box$}
\makeatletter
\renewenvironment{proof}[1][\proofname]{\par
  \normalfont
  \topsep6\p@\@plus6\p@ \trivlist
  \item[\hskip\labelsep{\bfseries #1}\@addpunct{\bfseries.}]\ignorespaces
}{
  \endtrivlist
}
\renewcommand{\proofname}{Proof}
\makeatother

\theoremstyle{plain}
\newtheorem{defi}{Definition}[section]
\newtheorem{lem}[defi]{Lemma}
\newtheorem{thm}[defi]{Theorem}
\newtheorem{rem}[defi]{Remark}

\newtheorem{prop}[defi]{Proposition}
\newtheorem{cor}[defi]{Corollary}

\newtheorem{assumption}[defi]{Assumption}

\newcommand{\tr}{\mathop{\rm tr}}

\newcommand{\Var}{\mathop{\rm Var}\nolimits}

\title{Moment convergence of the generalized maximum composite likelihood estimators for determinantal point processes}
\author{Kou Fujimori$^{*}$;\quad Sota Sakamoto$^{\dagger}$;\quad Yasutaka Shimizu$^{\ddagger}$\vspace{2mm}\\
$*$\,{\it Department of Mathematics} \\
$\dagger$\, {\it Graduate School of Fundamental Science and Engineering}\\
$\ddagger$\,{\it Department of Applied Mathematics} \\ {\it Waseda University}}
\date{\today}

\begin{document}
\maketitle

\begin{abstract}
The maximum composite likelihood estimator for parametric models of determinantal point processes (DPPs) is discussed. Since the joint intensities of these point processes are given by determinant of positive definite kernels, we have the explicit form of the joint intensities for every order. This fact enables us to consider the generalized maximum composite likelihood estimator for any order. 
This paper introduces the two step generalized composite likelihood estimator
and shows the moment convergence of the estimator under a stationarity.
Moreover, our results can yield information criteria for statistical model selection within DPPs.
\begin{flushleft}
 \vspace{1mm}
\end{flushleft}
\end{abstract}

\section{Introduction}
Determinantal point processes (DPPs) are the classes of spatial point processes with repulsive properties for each pair of realized points.
DPPs are proposed by \cite{macchi} and studied statistical physics to capture the behavior of fermions.
Moreover, they are studied intensively in the field of probability by, for example, \cite{soshnikov2000}, \cite{hough2009} among others.
In particular, ergodic property and mixing conditions for stationary and isotropic DPPs are studied by \cite{soshnikov2000} and \cite{brillinger mixing}, which enable us to construct asymptotic theories for statistical models of DPPs.

Statistical inferences for DPPs are also studied by some authors.
\cite{lavancier16} proposed some parametric 
models for stationary DPPs and they also provide 
algorithms for numerical simulation and estimation procedure based on the likelihood method and the minimum contrast method by using Ripley's $K$-functions and pairwise correlation functions for these parametric models.
Asymptotic properties of the minimum contrast estimator (MCE) is studied by \cite{cic dpp}. They proved the consistency and the asymptotic normality for MCEs based on the $K$-function and pairwise correlation function.
These estimation procedures are widely used in real data analysis such as the mobile network, the machine learning among others.
However, the theoretical properties for the maximum likelihood estimators for DPPs have not been studied enough up to our knowledge.

In this paper, we introduce the two step maximum composite likelihood estimator for parametric models of DPPs. 
The two step estimation methods are studied by, for example,  \cite{waagepetersen2009} for inhomogeneous spatial point processes.
They constructed the estimators by using second order estimating functions including the composite likelihood functions.
Composite likelihood approach studied by, e.g., \cite{guan2006} is widely used in several applications.
Since well-known point processes such as the Gibbs point process or 
Cox process only have the exact form of joint intensities of second order, 
we often use the second order composite likelihood function.
On the other hand, joint intensities of DPPs are given by the determinant of positive 
definite kernels, which allows us to compute the joint intensities of any order. 
Therefore, we can consider generalized maximum composite likelihood estimator
 (GMCLE) by using $p$-th order joint intensity for every integer $p \geq 2$. 
In this paper, we prove 
the consistency, the asymptotic normality and the moment convergence of the
estimator for stationary DPPs by using polynomial type large deviation inequality which is introduced by \cite{yoshida11}.
In particular, the moment convergence enables us to derive an information criterion.

This paper is organized as follows.
In section \ref{setups}, we provide setups for parametric models of DPPs and some regularity conditions.
We present the definition of the two step generalized maximum composite likelihood estimator and the consistency results for stationary case in Section \ref{estimation}.
Moreover, we discuss the moment convergence of the estimator in 
Section \ref{pldi}.
In Section \ref{numerical}, we show the finite sample performance of 
the second order estimator for well-known parametric models of DPPs, 
which are given by Gaussian, Laplace and Cauchy kernels.

Hereafter, for every $v=(v_1,\ldots,v_k)^\top \in \mathbb{R}^k$,\ $k =1,2\ldots$, we denote by
\[
|v| = \left( \sum_{i = 1}^k v_i^2
\right)^{\frac{1}{2}}.
\]
Similarly, for every tensor $u =(u_{i_1,\ldots,i_m})_{i_1,\ldots,i_m} \in (\mathbb{R}^k)^{\otimes m}$,
we denote that
\[
|u| = \left( \sum_{i_1,\ldots,i_m = 1}^k u_{i_1\ldots,i_m}^2
\right)^{\frac{1}{2}}.
\]
Moreover, for every smooth function 
$f:\ \mathbb{R}^q \rightarrow \mathbb{R}$
we denote its derivatives as follows:
\[
\partial_{\theta_i}f(\theta) = \frac{\partial}{\partial \theta_i} f(\theta),\quad i=1,\ldots,q,
\]
\[
\partial_\theta f(\theta) = \left(\partial_{\theta_1}f(\theta),\ldots,\partial_{\theta_q}f(\theta)\right)^\top \in \mathbb{R}^q,
\]
\[
\partial^2_\theta f(\theta) = \left(\partial_{\theta_i} \partial_{\theta_j} f(\theta)\right)_{i,j} \in \mathbb{R}^{q \times q}
\]
and 
\[
\partial^3_\theta f(\theta) = \left(\partial_{\theta_i} \partial_{\theta_j} \partial_{\theta_k} f(\theta)\right)_{i,j,k} \in \left(\mathbb{R}^{q}\right)^{\otimes 3}.
\]
\section{Determinantal point processes}\label{setups}
Let $(\Omega, \mathcal{F}, \mathbb{P})$ be a probability space
and $\mathcal{B}_0 (\mathbb{R}^d)$ be class of bounded Borel sets on 
$\mathbb{R}^d$.
Set valued function $X: \Omega \rightarrow \mathcal{N}^d$ is called a
$d$-dimensional point process, where
\[
\mathcal{N}^d := \left\{\mathcal{X} \subset \mathbb{R}^d | 
\# (\mathcal{X} \cap B) < \infty,\ B \in \mathcal{B}_0 (\mathbb{R}^d)\right\}.
\] 
For the kernel function $K: \mathbb{R}^{d}\times \mathbb{R}^d \rightarrow \mathbb{R}$,
the point process $X$ called determinantal point process with kernel $K$ ($X \sim DPP(K)$) if the measure; 
\[
\mathbb{E}\left[\sum_{(x_1,\ldots,x_p) \in X^p}^{\not=} 1_{\{x_1 \in A_1,\ldots, x_p \in A_p\}}\right],\quad p=1,2,\ldots
\]
where $A_j,\ j = 1,\ldots,p$ are bounded Borel sets on $\mathbb{R}^d$ and $X^p$ is $p$-direct product of the point process $X$ and the symbol $\sum_{(x_1,\ldots,x_p) \in X^p}^{\not =}$ means the summation over the mutually disjoints $p$-point $x_1,\ldots,x_p$, has the following density function
\[
\rho^{(p)}(x_1,\ldots,x_p) = \det[K](x_1,\ldots,x_p),
\]
with
\[
[K](x_1,\ldots,x_p):=\left(K(x_i, x_j)\right)_{1 \leq i,j \leq p} \in \mathbb{R}^{p \times p}.
\]
We call the function $\rho^{(1)}(x),\ x \in \mathbb{R}^d$
an intensity function and 
the function
\[
g(x,y):=\frac{\rho^{(2)}(x,y)}{\rho^{(1)}(x) \rho^{(1)}(y)},\quad x, y \in \mathbb{R}^d
\]
a pair correlation function respectively.
As well as \cite{lavancier18}, 
we consider the case when the kernel $K$ satisfies the following form:
\begin{equation}\label{model}
K(x,y) = \sqrt{\rho(x, \lambda) \rho(y, \lambda)} C_{\alpha}(x-y),
\end{equation}
where the function $\rho(\cdot,\lambda)$ is the intensity function with an unknown parameter $\lambda >0$ and $C_\alpha: \mathbb{R}^d \rightarrow \mathbb{R}$ is the function which satisfies $C_\alpha(0)=1$ with an unknown parameter $\alpha \in \mathbb{R}^q$. 
We write $\theta = (\lambda, \alpha)$ and the parameter space 
$\Theta = \Theta_\lambda \times \Theta_\alpha \subset \mathbb{R} \times \mathbb{R}^q$ and denote 
the kernel 
$K = K_\theta$, where $\theta \in \Theta$.
The following assumption ensures the unique existence of the DPP with kernel 
$K_\theta$. See, e.g., \cite{lavancier16} for details.
\begin{assumption}\label{krho}
The parametric model of the determinantal point processes given by 
(\ref{model})
satisfies the following conditions.
\begin{itemize}
\item[(i)]
For every $\theta \in \Theta$, the function $\rho(\cdot, \lambda)$ is bounded and 
$C_\alpha \in L^2(\mathbb{R}^d)$.
\item[(ii)]
For every $\theta \in \Theta$, the function 
$C_\alpha$ has a spectral density bounded by $1/\|\rho(\cdot, \lambda)\|_\infty$.
\end{itemize}
\end{assumption}
In particular, suppose kernel $K_\theta$ of the DPP $X$ 
has the following form, 
\[
K_\theta(x,y) = \lambda C_\alpha(x-y),
\]
which implies that the DPP $X$ is stationary.
In this case, the above condition (ii) is equivalent to that
the Fourier transformation $\mathcal{F} K_\theta$ satisfies the following inequality for every $\theta \in \Theta$
\[
0 \leq \mathcal{F}K_{\theta} \leq 1.
\]

Moreover, we make the following assumptions in order to establish asymptotic theories.
\begin{assumption}\label{regularity}
The following conditions hold true.
\begin{itemize}
\item[(i)]
There exists some $r>0$ such that, for every $(x,y)$ with $|x-y| \leq r$
\[
K_{\theta_1}(x,y) = K_{\theta_2}(x,y) \Rightarrow \theta_1 = \theta_2.
\]
\item[(ii)]
The kernel $K_\theta$ is positive definite.
Moreover, it holds that 
\[
\inf_{(x_1,\ldots,x_p)}\det[K](x_1,\ldots,x_p) >0, p \geq 1.
\]
\item[(iii)]
The kernel $K_\theta$ is fourth continuously differentiable with respect to $\theta$ and all the derivatives are bounded.
\item[(iv)]
The parameter space $\Theta$ is a compact and convex subset of 
$\Theta \times \mathbb{R}^q$.
Moreover, $\Theta$ admits the Sobolev's inequality for imbedding $W^{1,k} \hookrightarrow C(\Theta)\ (k \geq 1)$, i.e., for every $f \in W^{1,k}(\Theta) \subset C(\Theta)$, it holds that 
\[
\sup_{\theta \in \Theta} |f(\theta)| \lesssim \left(\|f\|_{L_k(\Theta)} + \||\partial_\theta f|\|_{L_k(\Theta)}\right),
\]
where 
\[\|f\|_{L_k(\Theta)} := \left(\int_\Theta |f(\theta)|^k d\theta\right)^{\frac{1}{k}}.
\]
\end{itemize}
\end{assumption}
See, e.g., \cite{adams and fournier} for conditions to ensure Sobolev's inequality.

\section{Two step generalized maximum composite likelihood estimator}\label{estimation}
\subsection{Estimation method}
Let $X$ be a $d$-dimensional determinantal point process with kernel 
$K_\theta$ satisfying Assumption  
\ref{krho}, where $\theta \in \Theta$ is an unknown parameter and 
the parameter space $\Theta$ is compact.
Suppose that there exists
the true value 
$\theta_0 = (\lambda_0, \alpha_0)$ in an interior of $\Theta$.
Our goal is to construct asymptotically normal estimators for $\theta_0$.
For the intensity parameter $\lambda$, we consider the following 
normalized quasi-likelihood function which can be seen in, e.g., \cite{clinet and yoshida}:
\[
\mathbb{H}_{n1}(\lambda) = 
\int_{D_n} \log \rho(x, \lambda) N(dx) - 
\int_{D_n} \rho(u, \lambda) du,
\]
where $D_n \subset \mathbb{R}^d$ is the observation window centered $0$ which satisfies the following condition:
\[
|D_n| \asymp n^d,\quad \mu_{d-1}(\partial D_n) \asymp n^{d-1},\quad n \rightarrow \infty
\]
with $|\cdot|$ and $\mu_{d-1}(\cdot)$ is the $d$ and $d-1$-dimensional Lebesgue measure, respectively.
Inspired by \cite{guan2006}, we define the following normalized $p$-th order composite likelihood function
for every integer $p \geq 2$ to estimate the parameter $\alpha$:
\begin{eqnarray}\label{CL}
\mathbb{H}_{n2}^{(p)}(\lambda, \alpha) 
&:=& 
\int_{D_n^p} \left\{
\log [\rho_{\theta}^{(p)}(x_1,\ldots,x_p)] - \log[K_{w,p}(r:\theta)]\right\} \nonumber \\
&& \times w_r(x_1,\ldots,x_p) N^{(p)}(dx_1 \cdots dx_p),
\end{eqnarray} 
where $r > 0$ is a tuning parameter, $\rho_{\theta}^{(p)}$ and $K_{w,p}$ are respectively the joint intensity of $p$-th order of DPP($C_\theta$) and the modified $K$-function of $p$-th order: 
\begin{eqnarray*}
\rho_{\theta}^{(p)}(x_1,\ldots,x_p) 
&=& \det[K_\theta](x_1,\ldots,x_p), \\
&=& \prod_{i=1}^p \rho(x_i, \lambda) \det[C_\theta](x_1,\ldots,x_p)
\end{eqnarray*}
and 
\[
K_{w,p}(r:\theta) =
 \int_{D_n^p} \rho_{\theta}^{(p)} (x_1,\ldots,x_p) w_r(x_1,\ldots,x_p)dx_1\cdots x_p,
\]
where $w_r$ is a bounded weight function whose support is given by 
\[S_r^p := \{(x_1,\ldots,x_p) : |x_1 - x_j| \leq r,\ 1 \leq j \leq p\}.\]
For example, a simple choice of the weight function $w_r$ is given by 
\[
w_r(x_1,\ldots,x_p) = 1_{S_r^p}(x_1,\ldots,x_p).
\]
Note that $N^{(p)}$ is a counting measure of $p$-th order induced by the 
point process $X$, i.e., 
\[
N^{(p)} \left(\prod_{j=1}^p A_j \right) = \sum_{(x_1,\ldots,x_p) \in X^p}^{\not=} 1_{\{x_1 \in A_1,\ldots, x_p \in A_p\}},
\]
where $A_j,\ j = 1,\ldots,p$ are bounded Borel sets on $\mathbb{R}^d$.
Using the estimating functions $\mathbb{H}_{n1}(\lambda)$ and 
$\mathbb{H}_{n2}(\lambda, \alpha)$, we define the following two-step estimator for $\theta = (\lambda, \alpha)$.
\begin{defi}\label{def of estimator}
The estimator $\hat{\theta}_n^{(p)} = (\hat{\lambda}_n, \hat{\alpha}_n^{(p)})$ is called generalized maximum composite likelihood estimator
if 
\begin{eqnarray}\label{GMCLE}
\hat{\lambda}_n 
&:=& \arg \sup_{\lambda \in \Theta_\lambda} \mathbb{H}_{n1}(\lambda), \\
\hat{\alpha}_n^{(p)}
&:=& \arg \sup_{\alpha \in \Theta_\alpha} \mathbb{H}_{n2}^{(p)}(\hat{\lambda}_n, \alpha).
\end{eqnarray}
\end{defi}
We investigate the asymptotic behavior of this estimator for a stationary case.
\subsection{Stationarity of DPPs}
Hereafter, we assume that the kernel function 
$K_\theta(\cdot,\cdot)$ of DPP $X$ 
has the following form:
\begin{equation}\label{stationary kernel}
K_\theta(x,y) = \lambda C_\alpha(x-y),\quad x, y \in \mathbb{R}^d,
\end{equation}
in other words, intensity function is reduced to the constant 
$\lambda$ and $C_\alpha$ is translation invariant.
This condition implies that the DPP $X$ is stationary.
Under Assumption \ref{regularity}, the estimator $\hat{\theta}_n$ is given by the following estimating equation:
\[
U_n^{(p)}(\theta) = 0,
\]
where $U_n^{(p)}(\theta) = (\partial_\lambda \mathbb{H}_{n1}(\lambda), \partial_\alpha \mathbb{H}_{n2}^{(p)\top}(\lambda, \alpha))^\top$, where the symbol $\top$ stands for the transpose.
If the kernel $K_\theta$ satisfies (\ref{stationary kernel}), 
it holds that
\[
\mathbb{H}_{n1}(\lambda) = 
\int_{D_n} \log \lambda N(dx) - 
\int_{D_n} \lambda du.
\]
We write $U_{n1}(\lambda)$ for the derivative of $\mathbb{H}_{n1}(\lambda)$, i.e., 
\[
U_{n1}(\lambda) := \partial_\lambda \mathbb{H}_{n1}(\lambda) =
\int_{D_n} \frac{1}{\lambda} N(dx) - 
\int_{D_n} du.
\]
It is easily seen that the estimator $\hat{\lambda}_n$ is given 
explicitly as follows 
\[
\hat{\lambda}_n = \frac{1}{|D_n|} \sum_{x \in X} 1_{x \in D_n} (x),
\]
which is the maximum likelihood estimator for intensity parameter 
and it is well known that it satisfies the consistency 
and the asymptotic normality. See, e.g., \cite{brillinger mixing}. 
Moreover, the estimating function 
$\mathbb{H}_{n2}(\lambda, \alpha)$ is given as follows 
when the DPP is stationary:
\begin{eqnarray*}
\mathbb{H}_{n2}^{(p)}(\lambda, \alpha) 
&:=& 
\int_{D_n^p} \left\{
\log [\rho_{\theta}^{(p)}(x_1,\ldots,x_p)] - \log[K_{w,p}(r:\theta)]\right\}\\
&& \times w_r(x_1,\ldots,x_p) N^{(p)}(dx_1 \cdots dx_p) \\
&=& 
\int_{D_n^p} \left\{
\log [\lambda^p \tilde{\rho}_{\alpha}^{(p)}(x_1,\ldots,x_p)] - \log[\lambda^p\tilde{K}_{w,p}(r:\alpha)]\right\}\\
&& \times w_r(x_1,\ldots,x_p) N^{(p)}(dx_1 \cdots dx_p) \\
&=& 
\int_{D_n^p} \left\{
\log [\tilde{\rho}_{\alpha}^{(p)}(x_1,\ldots,x_p)] - \log[\tilde{K}_{w,p}(r:\alpha)]\right\}\\
&& \times w_r(x_1,\ldots,x_p) N^{(p)}(dx_1 \cdots dx_p),
\end{eqnarray*}
where 
\[
\tilde{\rho}_{\alpha}^{(p)}(x_1,\ldots x_p) = \det[C_\alpha](x_1,\ldots,x_p),\quad x_1,\ldots,x_p \in \mathbb{R}^d,
\]
\[
[C_\alpha(x_1,\ldots,x_p)] = \left(C_\alpha (x_i - x_j)\right)_{1 \leq i, j \leq p}
\]
and 
\[
\tilde{K}_{w,p}(r:\alpha) = \int_{D_n^p} \tilde{\rho}_{\alpha}^{(p)} (x_1,\ldots,x_p) w_r(x_1,\ldots,x_p)dx_1\cdots x_p.
\]
Since we can see that the estimating function 
$\mathbb{H}_{n2}$ does not depend on $\lambda$, we denote this function and the score function
\[
CL_n^{(p)}(\alpha)=
\mathbb{H}_{n2}(\lambda, \alpha); \quad
U_{n2}^{(p)}(\alpha)=
\partial_\alpha \mathbb{H}_{n2}(\lambda, \alpha).
\]

Since we assume that $\tilde{\rho}_{\alpha}^{(p)}$ is a continuously differentiable function with respect to $\alpha$ by Assumption \ref{regularity}, we have that
\begin{eqnarray}\label{Score}
U_{n2}^{(p)}(\alpha)
&:=& 
\int_{D_n^p} \left[
\frac{\partial_\alpha \tilde{\rho}_{\alpha}^{(p)}}{\tilde{\rho}_{\alpha}^{(p)}}(x_1,\ldots,x_p) - \frac{\partial_\alpha \tilde{K}_{w,p} (r:\alpha)}{\tilde{K}_{w,p}(r:\alpha)}
\right] \nonumber \\
&& \times w_r(x_1,\ldots,x_p) N^{(p)}(dx_1\cdots dx_p),
\end{eqnarray}
where $\partial_\alpha \tilde{\rho}_{\alpha}^{(p)}$ and $\partial_\alpha \tilde{K}_{w,p}$ is the first derivative with respect to 
$\alpha$ for each function.
Note that under Assumptions \ref{krho} and \ref{regularity}, $\hat{\theta}_n$ solves the following estimating equation:
\begin{equation}\label{GMCLE score}
U_{n2}^{(p)}(\hat{\alpha}_n^{(p)}) = 0.
\end{equation}
We shall prove the asymptotic behavior of this estimator.
The next theorem states that GMCLE satisfies the consistency.
\begin{thm}\label{consistency}
Under Assumptions \ref{krho} and \ref{regularity}, the estimator $\hat{\theta}_n$ is consistent to $\theta_0$, i.e., 
\[
\hat{\theta}_n^{(p)} \stackrel{\mathbb{P}}{\longrightarrow} \theta_0,\quad n \rightarrow \infty.
\]
\end{thm}
\begin{proof}
Since the consistency of $\hat{\lambda}_n$ is seen in, e.g., 
Biscio and Lavancier (2017), we shall prove the consistency of 
$\hat{\alpha}_n^{(p)}$ only.

It is well-known that$X$ is ergodic under Assumptions \ref{krho} and \ref{regularity}, (See e.g. Soshnikov (2000)).
Noting that 
\begin{eqnarray*}
\mathbb{E}_{\theta_0} \left[\frac{1}{|D_n|}U_{n2}^{(p)}(\alpha)\right]
&=& \frac{1}{|D_n|} \int_{D_n^p} \left[
\frac{\partial_{\alpha} \tilde{\rho}_{\alpha}^{(p)}}{\tilde{\rho}_{\alpha}^{(p)}}(x_1,\ldots,x_p) - \frac{\partial_{\alpha} \tilde{K}_{w,p} (r:\alpha)}{\tilde{K}_{w,p}(r:\alpha)}
\right] \\
&& \times w_r(x_1,\ldots,x_p) \rho_{\theta_0}^{(p)}(x_1,\ldots,x_p) dx_1\cdots dx_p,
\end{eqnarray*}
we have that $\mathbb{E}_{\theta_0} [U_{n,p}(\alpha_0)] = 0$ if and only if $\alpha = \alpha_0$.
Moreover, it follows from a direct calculation that
\[
\partial_\alpha \tilde{\rho}_{\alpha}^{(p)} (x_1,\ldots,x_p) = \tr [C_\alpha^{-1} \partial_\alpha C_{\alpha}](x_1,\ldots,x_p)
\]
and that 
\[
\partial_\alpha \tilde{K}_{w,p} (r;\alpha) = \int_{D_n^p} \tr [C_\alpha^{-1} \partial_\alpha C_{\alpha}](x_1,\ldots,x_p) w_r(x_1,\ldots,x_p) dx_1\cdots dx_p.
\]
Since under Assumption \ref{regularity}, $\det[C_\alpha](x_1,\ldots,x_p)$ is bounded by zero from below uniformly in 
$x_1,\ldots,x_p \in \mathbb{R}^d$, 
the parameter space $\Theta$ is assumed to be compact, and then the functions 
$\partial_\alpha \rho_{\alpha}^{(p)}$ and $\partial_\alpha K_{w,p}$ are continuous with respect to $\alpha$ and $(x_1,\ldots,x_p)$,
respectively,
it holds that the integrand in the score function $U_{n2}^{(p)}(\alpha)$ is bounded and continuous in $\alpha$ and $(x_1,\ldots,x_p)$.
We therefore obtain the conclusion as a consequence from Theorem 1 of \cite{guan2006}.
\qed
\end{proof}
Moreover, we can derive the asymptotic variance of the score function as follows.
The similar calculation is appeared in \cite{brillinger mixing} and 
\cite{lavancier18}.
\begin{prop}\label{var}
Under Assumptions \ref{krho} and \ref{regularity}, 
there exists the limit 
\[
\lim_{n \rightarrow \infty}
\Var_{\theta_0} \left[\frac{1}{\sqrt{|D_n|}} U_{n}^{(p)}(\theta_0)\right]=:\Sigma^{(p)}(\theta_0), 
\]
which is a $(q+1) \times (q+1)$ matrix.
\end{prop}
\begin{proof}
To show the existence of the matirx $\Sigma^{(p)}(\theta_0)$, we set
\begin{eqnarray}\label{variance matrix}
\Sigma^{(p)}(\theta_0) = \left(\begin{array}{cc}
\Sigma_{11}(\theta_0) & \Sigma_{12}^{(p)}(\theta_0) \\
\Sigma_{21}^{(p)}(\theta_0) & \Sigma_{22}^{(p)}(\theta_0) \\
\end{array}
\right),
\end{eqnarray}
where
\[
\Sigma_{11}(\theta_0) = 
\lim_{n \rightarrow \infty} \frac{1}{|D_n|} \mathbb{E}_{\theta_0}[\{U_{n1}(\lambda_0)\}^2],
\]
\[
\Sigma_{12}^{(p)}(\theta_0) = \Sigma_{21}^{(p)\top}(\theta_0) = 
\lim_{n \rightarrow \infty} \frac{1}{|D_n|} 
\mathbb{E}_{\theta_0} [U_{n1}(\lambda_0) U_{n2}^{(p)\top}(\alpha_0)]
\]
and 
\[
\Sigma_{22}^{(p)}(\theta_0) =
\lim_{n \rightarrow \infty} 
\frac{1}{|D_n|} \mathbb{E}_{\theta_0} [U_{n2}^{(p)}(\alpha_0) U_{n2}^{(p)\top}(\alpha_0)].
\]
The component $\Sigma_{11}(\theta_0)$ is written explicitly as
\[
\Sigma_{11}(\theta_0) = \frac{1}{\lambda_0} - \int_{\mathbb{R}^d} C_{\alpha_0}^2(u) du.
\]
As for the other components, we shall only show the existence of $\Sigma_{22}^{(p)}(\theta_0)$ since the others are similarly proved. 

Note that, for every $i,j \in \{1,2,\ldots,q\}$, 
\begin{eqnarray*}
\lefteqn{
\frac{1}{|D_n|} \mathbb{E}_{\theta_0} [U_{n2}^{(p) i}(\alpha_0)U_{n2}^{(p) j}(\alpha_0)]}\\
&=&
\frac{1}{|D_n|} \mathbb{E}_{\theta_0}\left[
\sum_{(x_1,\ldots,x_p) \in X^p}^{\not =} U_2^{(p)i}(x_1,\ldots,x_p;\alpha_0)
\sum_{(y_1,\ldots,y_p) \in X^p}^{\not =} U_2^{(p)j}(y_1,\ldots,y_p;\alpha_0), 
\right]
\end{eqnarray*}
where 
\begin{eqnarray*}
U_2^{(p)i}(x_1,\ldots,x_p;\alpha_0) 
&=&\left(\frac{\partial_{\alpha_i} \tilde{\rho}_{\alpha_0}^{(p)}}{\tilde{\rho}_{\alpha_0}^{(p)}}(x_1,\ldots,x_p) - \frac{\partial_{\alpha_i} \tilde{K}_{w,p} (r:\alpha_0)}{\tilde{K}_{w,p}(r:\alpha_0)}\right)\\
&\times& w_r(x_1,\ldots,x_p).
\end{eqnarray*}
For every $p \leq l \leq 2p$, assume that 
$x_1,\ldots x_p$ and $y_1,\ldots y_p$ have $2p-l$ common variables.
Denote the $l$-different variables by $x_{(1)},\ldots,x_{(l)}$.
Since $C_\alpha(0) = 1$, the expectation is essentially a function of 
$x_{(1)},\ldots,x_{(l)}$, which can be written by the finite sum of the following integrals
for each $l = p,\ldots,2p$ and all combination of such different variables :
\begin{eqnarray*}
&& \frac{1}{|D_n|} \int_{D_n^l}G_{\alpha_0}^{i,j} (x_{(1)},\ldots,x_{(l)}) W_r(x_{(1)},\ldots,x_{(l)})\\
&&
\ \ \ \ \ \ \ \ \ \ \ \ \times \rho_{\theta_0}^{(l)}(x_{(1)},\ldots,x_{(l)}) dx_{(1)} \cdots dx_{(l)},
\end{eqnarray*}
where 
\begin{eqnarray*}
\lefteqn{
G_{\alpha_0}^{i,j}(x_{(1)},\ldots,x_{(l)})}\\
&=&\left\{
\frac{\partial_{\alpha_i}\tilde{\rho}_{\alpha}^{(p)}}{\tilde{\rho}_{\alpha}^{(p)}}(x_1,\ldots,x_p) - \frac{\partial_{\alpha_i}\tilde{K}_{w,p} (r:\alpha)}{\tilde{K}_{w,p}(r:\alpha)}
\right\}\\
&&\times \left\{\frac{\partial_{\alpha_j}\tilde{\rho}_{\alpha}^{(p)}}{\tilde{\rho}_{\alpha}^{(p)}}(y_1,\ldots,y_p) - \frac{\partial_{\alpha_j}\tilde{K}_{w,p} (r:\alpha)}{\tilde{K}_{w,p}(r:\alpha)}\right\}
\end{eqnarray*}
and 
\begin{eqnarray*}
W_r(x_{(1)},\ldots,x_{(l)})
= w_r(x_1,\ldots,x_p) \times w_r(y_1,\ldots,y_p),
\end{eqnarray*}
under the assumption that $x_1,\ldots,x_p$ and $y_1,\ldots,y_p$ have $l$-different variables.
We change the variables by $u_1 = x_{(1)}$ and $u_k = x_{(k)}-x_{(1)},\ k = 2,\ldots,l$. 
Since we have that
\[\tilde{\rho}_{\alpha_0}^{(l)} (x_{(1)},\ldots,x_{(l)})
= \det \left[C_{\alpha_0}\right](x_{(1)},\ldots,x_{(l)})
\]
and
\[
\partial_\alpha \tilde{\rho}_{\alpha_0}^{(p)} = \tr [C_{\alpha_0}^{-1} \partial_\alpha C_{\alpha_0}] (x_{(1)},\ldots,x_{(l)}),
\]
we can easily see that there exist the functions $\bar{\rho}_{\alpha_0,l}(u_1,u_2,\ldots,u_l)$ and 
$H_{\alpha_0}^{i,j}(u_1,u_2,\ldots,u_l)$ such that 
\begin{eqnarray*}
H_{\alpha_0}^{i,j}(0,u_2,\ldots,u_l) = G_{\alpha_0}^{i,j}(x_{(1)},\ldots,x_{(l)})
\end{eqnarray*}
and 
\begin{eqnarray*}
\bar{\rho}_{\theta_0}^{(l)}(0,u_2,\ldots,u_l) = {\rho}_{\theta_0}^{(l)}(x_{(1)},\ldots,x_{(l)}).
\end{eqnarray*}
Consider the following sets: 
\begin{eqnarray*}
D_n^{\ominus r} 
&:=& \{x \in D_n | |x-y| > r,\ y \in \partial D_n\}, \\
D_n-u_1 
&:=& \left\{
x \in \mathbb{R}^d | x= v - u_1,v \in D_n\right\}, \\
S_{r,0}^{l} 
&:=& \{(0, u_2,\ldots, u_l) | |u_j| \leq r, j=2,\ldots,l\},
\end{eqnarray*}
where $S_{r,0}^l$ is the support set of the function $W_r(0,u_2,\ldots,u_l)$.
Noticing that, if $u_1 \in D_n^{\ominus r}$, it holds that 
$S_{r,0}^l \subset (D_n- u_1)^{l-1}$, we have
\begin{eqnarray*}
\lefteqn{
\frac{1}{|D_n|} \mathbb{E}_{\theta_0} [U_{n,p}^i(\alpha_0)U_{n,p}^j(\alpha_0)]}\\
&=& \frac{1}{|D_n|} \int_{D_n} \int_{D_n-u_1}\cdots \int_{D_n-u_1}
H_{\alpha_0}^{i,j}(0,u_2,\ldots,u_l)\\
&&\ \times W_r(0,u_2,\ldots,u_l) \bar{\rho}_{\theta_0}^{(l)}(0,u_2,\ldots,u_l)
du_1\cdots du_l \\
&=& (I) + (II), 
\end{eqnarray*}
where 
\begin{eqnarray*}
(I) 
&=&  \frac{1}{|D_n|} \int_{D_n^{\ominus r}} \int_{S_{r,0}^l}
H_{\alpha_0}^{i,j}(0,u_2,\ldots,u_l)\\
&&\ \times W_r(0,u_2,\ldots,u_l) \bar{\rho}_{\theta_0}^{(l)}(0,u_2,\ldots,u_l)
du_1\cdots du_l 
\end{eqnarray*}
and 
\begin{eqnarray*}
(II)
&=& \frac{1}{|D_n|} \int_{D_n\setminus D_n^{\ominus}} \int_{(D_n-u_1)^{l-1}}
H_{\alpha_0}^{i,j}(0,u_2,\ldots,u_l)\\
&&\ \times W_r(0,u_2,\ldots,u_l) \bar{\rho}_{\theta_0}^{(l)}(0,u_2,\ldots,u_l)
du_1\cdots du_l.
\end{eqnarray*}
We can evaluate $|(II)|$ as follows:
\begin{eqnarray*}
|(II)|
&\leq& \frac{1}{|D_n|} \int_{D_n\setminus D_n^{\ominus}} \int_{(\mathbb{R}^d)^{l-1}} \left|
H_{\alpha_0}^{i,j}(0,u_2,\ldots,u_l) \right. \\
&&\left. \times W_r(0,u_2,\ldots,u_l) \bar{\rho}_{\theta_0}^{(l)}(0,u_2,\ldots,u_l) \right|
du_1\cdots du_l \\
&\leq&
\frac{|D_n \setminus D_n^{\ominus r}|}{|D_n|} \|H_{\alpha_0}^{ij}\|_\infty \int_{(\mathbb{R}^d)^{l-1}} W_r(0,u_2,\ldots,u_l) \\
&&\times|\bar{\rho}_{\theta_0}^{(l)}(0,u_2,\ldots,u_l)|
du_2\cdots du_l.
\end{eqnarray*}
The right-hand side converges to $0$ as $n \rightarrow \infty$ since 
it follows from Assumption \ref{regularity} that
$|D_n \setminus D_n^{\ominus r}|/|D_n| \rightarrow 0$.
Moreover, noticing that
$|D_n^{\ominus r}/|D_n| \rightarrow 1$ as $n \rightarrow \infty$, we obtain
\begin{eqnarray*}
(I)
&=& \frac{|D_n^{\ominus r}|}{|D_n|}  \int_{S_{r,0}^l}
H_{\alpha_0}^{i,j}(0,u_2,\ldots,u_l)\\
&&\ \times W_r(0,u_2,\ldots,u_l) \bar{\rho}_{\theta_0}^{(l)}(0,u_2,\ldots,u_l)
du_1\cdots du_l  \\
&\rightarrow& \int_{S_{r,0}^l}
H_{\alpha_0}^{i,j}(0,u_2,\ldots,u_l)\\
&&\ \times W_r(0,u_2,\ldots,u_l) \bar{\rho}_{\theta_0}^{(l)}(0,u_2,\ldots,u_l)
du_1\cdots du_l \\
&=& \int_{\mathbb{R}^{l-1}}
H_{\alpha_0}^{i,j}(0,u_2,\ldots,u_l)\\
&&\ \times W_r(0,u_2,\ldots,u_l) \bar{\rho}_{\theta_0}^{(l)}(0,u_2,\ldots,u_l)
du_1\cdots du_l,\quad n \rightarrow \infty,
\end{eqnarray*}
which implies the conclusion.
\qed
\end{proof}
The next lemma states the convergence of the Hessian matrix of composite likelihood function.
\begin{lem}\label{fisher conv}
It holds that
\begin{equation}\label{unif conv}
-\frac{1}{|D_n|}\partial_\theta U_{n}^{(p)}(\hat{\theta}_n) \rightarrow^p \mathcal{I}^{(p)}(\theta_0)
\end{equation}
as $n \rightarrow \infty$, where $\mathcal{I}^{(p)}(\theta_0)$ is a $(q+1) \times (q+1)$ matrix 
of the following form.
\[
\mathcal{I}^{(p)}(\theta_0) = \left(\begin{array}{cc}
\frac{1}{\lambda_0} & 0 \\
0 & \mathcal{I}_{22}^{(p)}(\theta_0) \\
\end{array}
\right),
\]
and
\[
\mathcal{I}_{22}^{(p)}(\theta_0) = -\int_{(\mathbb{R}^d)^{p-1}} \partial_\alpha U_p(0,u_2\ldots,u_p;\alpha_0) \rho_{\theta_0,p}(0,u_2,\ldots,u_p) du_2\cdots du_p.
\]
\end{lem}
\begin{proof}
Under Assumption \ref{regularity}, the functions
\[
\partial^l_\alpha \rho_{\alpha}^{(p)}(x_1,\ldots,x_p),\quad l =0,1,2
\]
are bounded and continuous functions in 
$\alpha$ and $(x_1,\ldots,x_p)$.
In addition, it is obvious that $\rho_{\alpha}^{(p)}(x_1,\ldots,x_p)$ and $K_{w,p}(r;\alpha)$ is bounded below from zero uniformly in $\alpha$ and $(x_1,\ldots,x_p)$.
These properties are sufficient condition to show that
\begin{eqnarray*}
\sup_{|\alpha_1 - \alpha_2| < \delta}
|U_p(x_1,\ldots,x_p;\alpha_1)-U_p(x_1,\ldots,x_p;\alpha_2) |w_r(x_1,\ldots,x_p)
\end{eqnarray*}
converges to $0$ as $\delta \rightarrow 0$.
We therefore obtain the convergence
(\ref{unif conv}) 
in a similar manner to \cite{guan2006}, 
which ends the proof.
\qed
\end{proof}
\section{Moment convergence of the estimator}\label{pldi}
\subsection{Moment convergence of the estimator}
In this section, we will show the moment convergence of the estimator $\hat{\lambda}_n$ and $\hat{\alpha}_n$ under a stationarity by using a polynomial large deviation inequality established by \cite{yoshida11}.
Noticing that under Assumptions \ref{krho} and \ref{regularity}, it follows from the ergodicity that there exists the limit 
$\widetilde{CL}^{(p)}$ 
\[
\widetilde{CL}^{(p)}(\alpha)= \lim_{n \rightarrow \infty} \mathbb{E}_{\theta_0}\left[
\frac{1}{|D_n|} CL_n^{(p)}(\alpha)\right]
\]
such that 
\[
\widetilde{CL}^{(p)}(\alpha)
=
\lim_{n \rightarrow \infty} \frac{1}{|D_n|} CL_n^{(p)}(\alpha),\quad a.s.,
\]
for every $\alpha \in \Theta_\alpha$.
The moment convergence of the estimator is given as follows.
\begin{thm}\label{pldi DPP}
Suppose that Assumptions \ref{krho} and \ref{regularity} are fulfilled 
and the matrix $\mathcal{I}_{22}^{(p)}$ 
defined in Lemma \ref{fisher conv}
is positive definite.
It holds for every polynomial growth function $f: \mathbb{R}^{q+1} \rightarrow \mathbb{R}$, the random sequence $\hat{u}_n = \sqrt{|D_n|}(\hat{\theta}_n - \theta_0)$ and the random variable 
$u \sim N\left(0, \mathcal{I}^{(p)-1}(\theta_0) \Sigma^{(p)}(\theta_0) \mathcal{I}^{(p)-1}(\theta_0)\right)$ that
\[
\lim_{n \rightarrow \infty} \mathbb{E}_{\theta_0}\left[f(\hat{u}_n)\right] = \mathbb{E}_{\theta_0}\left[f(u)\right].
\]
\end{thm}
To prove this theorem, it suffices to verify the following conditions (M1)-(M4) for any integer $L \geq 2 $.
See, e.g., \cite{yoshida11} or \cite{SZ2017} for details. 
\begin{itemize}
\item[(M1)]
The following inequalities hold true.
\[
\sup_{n \in \mathbb{N}} \mathbb{E}_{\theta_0} \left[
\left|\frac{1}{\sqrt{|D_n|}} \partial_\alpha CL_n^{(p)}(\alpha_0)\right|^{L}
\right] < \infty, 
\]
and 
\[
\sup_{n \in \mathbb{N}} \mathbb{E}_{\theta_0} \left[
\left|\sqrt{|D_n|} \left(Y_{n2}(\alpha, \alpha_0)-Y_2(\alpha, \alpha_0)\right)\right|^{L}
\right] < \infty,
\]
where 
\begin{eqnarray*}
Y_{n2}(\alpha, \alpha_0) &=& \frac{1}{|D_n|} CL_n^{(p)}(\alpha) - \frac{1}{|D_n|} CL_n^{(p)}(\alpha_0), \\
Y_2(\alpha, \alpha_0) &=& \widetilde{CL}^{(p)}(\alpha) - \widetilde{CL}^{(p)}(\alpha_0).
\end{eqnarray*}
\item[(M2)]
For $M > 0$, it holds that
\[
\sup_{n \in \mathbb{N}} \mathbb{E}_{\theta_0}\left[
\left(
\sup_{\alpha \in \Theta_\alpha} \frac{1}{|D_n|} \left|
\partial_\alpha^3 CL_n^{(p)}(\alpha)
\right|
\right)^M\right] < \infty, 
\]
and 
\[
\sup_{n \in \mathbb{N}} \mathbb{E}_{\theta_0}\left[
\left|
\sqrt{|D_n|}\left(
\mathcal{I}_{n22}^{(p)}(\alpha_0) - \mathcal{I}_{22}^{(p)}(\alpha_0)
\right)
\right|^{L}
\right] < \infty,
\]
where 
\[
\mathcal{I}_{n22}^{(p)}(\alpha) = - \partial^2_\alpha 
\left(\frac{1}{|D_n|} CL_n^{(p)}(\alpha)\right),\quad
\mathcal{I}_{22}^{(p)}(\alpha) = - \partial^2_\alpha \left(
\widetilde{CL}^{(p)}(\alpha)
\right).
\]
\item[(M3)]
The matrix $\mathcal{I}_{22}^{(p)}(\alpha_0)$ is positive definite.
\item[(M4)]
For every $\alpha \in \Theta_\alpha$, it holds that 
\[
Y_2(\alpha, \alpha_0) = Y_2(\alpha, \alpha_0) - Y_2(\alpha_0, \alpha_0) \lesssim -|\alpha - \alpha_0|^2.
\]
\end{itemize}
\begin{proof}[Proof of Theorem \ref{pldi DPP}]
It holds that 
\begin{eqnarray*}
\lefteqn{
\mathbb{E}_{\theta_0} \left[
\left|\frac{1}{\sqrt{|D_n|}} \partial_\alpha CL_n^{(p)}(\alpha_0)\right|^{L}
\right]} \\
&=&
|D_n|^{-\frac{L}{2}} \mathbb{E}_{\theta_0}\left[\left|
U_{n2}^{(p)}(\alpha_0)\right|^{L}
\right] \\
&=&
|D_n|^{-\frac{L}{2}} \int_{D_n^p} \left|
U_2^{(p)}(x_1,\ldots,x_p)\right|^L\lambda_0^p \tilde{\rho}_{\alpha_0}(x_1,\ldots,x_p)
dx_1\cdots dx_p
 \\
&\lesssim&
|D_n|^{1-\frac{L}{2}},
\end{eqnarray*}
which implies the first inequality in (M1).

To prove the second inequality, it is sufficient to prove that 
\[
\sup_{n \in \mathbb{N}} \mathbb{E}_{\theta_0} \left[\left(
\sup_{\alpha \in \Theta_\alpha} |D_n|^{\frac{L}{2}} 
\left|
\frac{1}{|D_n|} CL_n^{(p)}(\alpha) - \widetilde{CL}^{(p)}(\alpha)
\right|^{L} 
\right)\right] < \infty.
\]
By Sobolev's inequality in Assumption \ref{regularity}, 
we have that 
\begin{eqnarray*}
\lefteqn{
\mathbb{E}_{\theta_0} \left[\left(
\sup_{\alpha \in \Theta_\alpha} |D_n|^{\frac{L}{2}} 
\left|
\frac{1}{|D_n|} CL_n^{(p)}(\alpha) - \widetilde{CL}^{(p)}(\alpha)
\right|^{L} 
\right)\right]} \\
&\lesssim&
|D_n|^{\frac{L}{2}} \int_{\Theta_\alpha}
\mathbb{E}_{\theta_0} \left[
\left|
\frac{1}{|D_n|} CL_n^{(p)}(\alpha) - \widetilde{CL}^{(p)}(\alpha)
\right|^{L} 
\right] d\alpha \\
&&\ + |D_n|^{\frac{L}{2}} \int_{\Theta_\alpha}
\mathbb{E}_{\theta_0} \left[
\left|
\frac{1}{|D_n|} \partial_\alpha CL_n^{(p)}(\alpha) - \partial_\alpha \widetilde{CL}^{(p)}(\alpha)
\right|^{L} 
\right] d\alpha.
\end{eqnarray*}
Therefore, it holds for every $\alpha$ that 
\begin{eqnarray*}
\lefteqn{|D_n|^{\frac{L}{2}}
\mathbb{E}_{\theta_0}\left[\left|
\frac{1}{|D_n|} CL_n^{(p)}(\alpha) - \widetilde{CL}^{(p)}(\alpha)
\right|^{L}\right]} \\
&\lesssim&
|D_n|^{\frac{L}{2}}
\mathbb{E}_{\theta_0}\left[\left|
\frac{1}{|D_n|} CL_n^{(p)}(\alpha) - \mathbb{E}_{\theta_0}\left[\frac{1}{|D_n|}{CL_n}^{(p)}(\alpha)\right]
\right|^{L}\right] \\
&&\ + |D_n|^{\frac{L}{2}}
\left|
\mathbb{E}_{\theta_0}\left[\frac{1}{|D_n|}{CL_n}^{(p)}(\alpha)\right] - \widetilde{CL}^{(p)}(\alpha)
\right|^{L}.
\end{eqnarray*}
The first term in the right-hand side is $O(|D_n|^{1-L/2})$ by the Brillinger mixing condition for $X$.
As for second term, we have that 
\begin{eqnarray*}
\lefteqn{
\left|
\mathbb{E}_{\theta_0}\left[\frac{1}{|D_n|}{CL_n}^{(p)}(\alpha)\right] - \widetilde{CL}^{(p)}(\alpha)
\right|^{L}} \\
&\lesssim& |D_n|^{\frac{L}{2}}
\left|
\int_{S_r^p} \tilde{\rho}_{\alpha_0}^{(p)}(0, u_2,\ldots,u_p)
du_2\cdots du_p \right. \\
&& \left.
 - \frac{1}{|D_n|}\int_{D_n^{p} \cap S_r^{p}} \tilde{\rho}_{\alpha_0}^{(p)}(0, u_2,\ldots,u_p) du_1 \cdots du_p
\right|^{L}.
\end{eqnarray*}
Noticing that it holds that for $n \in \mathbb{N}$ large enough 
$S_r^p \subset D_n^p$, we have that the last term in the right-hand side is $0$.
Therefore, we can conclude that 
\[
\sup_{n \in \mathbb{N}}|D_n|^{\frac{L}{2}} \int_{\Theta_\alpha}
\mathbb{E}_{\theta_0} \left[
\left|
\frac{1}{|D_n|} CL_n^{(p)}(\alpha) - \widetilde{CL}^{(p)}(\alpha)
\right|^{L} 
\right] d\alpha < \infty
\]
since the parameter space $\Theta_\alpha$ is assumed to be compact.
Similarly, we have that 
\[\sup_{n \in \mathbb{N}}
|D_n|^{\frac{L}{2}} \int_{\Theta_\alpha}
\mathbb{E}_{\theta_0} \left[
\left|
\frac{1}{|D_n|} \partial_\alpha CL_n^{(p)}(\alpha) - \partial_\alpha \widetilde{CL}^{(p)}(\alpha)
\right|^{L} 
\right] d\alpha < \infty,
\]
which implies the second inequality in (M1).

Since condition (M2) can be verified similarly to (M1) and (M3) is one of the assumptions, 
we shall check the condition (M4).
From Taylor's expansion, we have that 
\begin{eqnarray*}
\lefteqn{
CL^{(p)}(\alpha) - CL^{(p)}(\alpha_0)}\\
&=& \lim_{n \rightarrow \infty} \mathbb{E}_{\theta_0}\left[\frac{1}{|D_n|}CL_n^{(p)}(\alpha) - \frac{1}{|D_n|}CL_n^{(p)}(\alpha_0)\right]\\
&=& \lim_{n \rightarrow \infty} \frac{1}{|D_n|}\left(\mathbb{E}_{\theta_0}\left[
\partial_\theta CL_n^{(p)}(\theta_0)^\top (\theta-\theta_0) + \frac{1}{2} (\theta - \theta_0)^\top \partial_\theta^2 CL_n^{(p)} (\tilde{\theta}) (\theta - \theta_0)
\right]\right)\\
&=& -\frac{1}{2} (\theta - \theta_0)^\top \mathcal{I}_{22}^{(p)}(\tilde{\alpha})(\alpha-\alpha_0)\\
&\lesssim& - |\alpha - \alpha_0|^2,
\end{eqnarray*}
where $\tilde{\alpha} = c \alpha + (1-c)\alpha_0$ for some $c \in [0,1]$.
This concludes (M4).
\qed
\end{proof}
We obtain the asymptotic normality and moment convergence of the estimator as a corollary of Theorem \ref{pldi DPP} as follows.
\begin{cor}\label{normality GCMLE}
Under the same assumptions as Theorem \ref{pldi DPP}, it holds that 
\[
\sqrt{|D_n|}(\hat{\theta}_n - \theta_0) \rightarrow^d N\left(0, \mathcal{I}^{(p)-1}(\theta_0) \Sigma^{(p)}(\theta_0) \mathcal{I}^{(p)-1}(\theta_0)\right)
\]
as $n \rightarrow \infty$.
Moreover, it holds that 
\[
\sup_{n \in \mathbb{N}}\mathbb{E}_{\theta_0}\left[
\left|\sqrt{|D_n|}\left(
\hat{\theta}_n^{(p)} - \theta_0\right)
\right|^L
\right] < \infty
\]
for every positive integer $L$.
\end{cor}
\subsection{Information criteria}\label{IC}
As application of the moment convergence of the estimator, we can derive an information criterion.
Since the explicit form of $\Sigma^{(p)}$ and $\mathcal{I}^{(p)}$ can be calculated, the bias of the estimator of $CL^{(p)}$ also can be calculated for every integer $p\geq2$, which yields the information criterion for the second order estimator $\hat{\alpha}_n^{(2)}$.
\begin{defi}\label{def IC}
An information criterion based on the second order composite likelihood is defined by 
\[
IC^{(2)} := -2CL_n^{(2)}(\hat{\alpha}_n^{(2)}) + 2\tr\left(\Sigma_{22}^{(2)}(\hat{\theta}_n^{(2)}) \mathcal{I}_{22}^{(2)-1}(\hat{\theta}_n^{(2)})\right),
\]
where $\hat{\theta}_n^{(2)}$ is the second order estimator for $\theta_0 = (\lambda_0, \alpha_0)$ and $\Sigma_{22}^{(2)}$ is defined in 
the proof of Proposition \ref{var}. 
\end{defi}
The information criteria $IC^{(2)}$ is an AIC-type criterion based on, not the true likelihood, but the composite likelihood. We can choose a model that has the smaller value of $IC^{(2)}$ in several competitive models. 

\begin{rem}
We can calculate the $\Sigma_{22}^{(2)}(\hat{\theta}_n^{(2)})$ and 
$\mathcal{I}_{22}^{(2)}(\hat{\theta}_n^{(2)})$ as follows
\begin{eqnarray}\label{asymptotic variance second}
\Sigma_{22}^{(2)}(\hat{\theta}_n)
&=& \iiint_{(\mathbb{R}^d)^3} U_2^{(2)}(0,u_2;\hat{\alpha}_n^{(2)}) U_2^{(2)\top}(u_3,u_4;\hat{\alpha}_n^{(2)})
\rho_{\hat{\theta}_n}^{(4)}(0,u_2,u_3,u_4)du_2 du_3 du_4 \nonumber \\
&& + 4 \iint_{(\mathbb{R}^d)^2} U_2^{(2)}(0,u_2;\hat{\alpha}_n^{(2)}) U_2^{(2)\top}(0,u_3;\hat{\alpha}_n^{(2)})
\rho_{\hat{\theta}_n^{(2)}}^{(3)}(0,u_2,u_3) du_2 du_3 \nonumber \\
&& + \int_{\mathbb{R}^d} U_2^{(2)}(0,u_2;\hat{\alpha}_n) U_2^{(2)\top}(0,u_2;\hat{\alpha}_n^{(2)}) \rho_{\hat{\theta}_n^{(2)}}^{(2)}(0,u_2) du_2.
\end{eqnarray}
\begin{eqnarray*}
\mathcal{I}_{22}^{(2)}(\hat{\theta}_n^{(2)})
&=& -\int_{\mathbb{R}^d} \partial_\theta U_2^{(2)}(0,u_2;\hat{\alpha}_n^{(2)}) \rho_{\hat{\theta}_n}^{(2)}(0,u_2) du_2.
\end{eqnarray*}
Note that since we can calculate $\Sigma^{(p)}(\hat{\theta}_n^{(p)})$ explicitly for every $p \geq 2$, we can define $IC^{(p)}$ similarly.
\end{rem}

 
It would be easy to understand the meaning of $IC^{(2)}$ if we consider the quantity 
\[
\widetilde{IC}^{(2)} := \frac{1}{|D_n|}CL_n^{(2)}(\hat{\alpha}_n^{(2)}) - \frac{1}{|D_n|}\tr\left(\Sigma_{22}^{(2)}(\theta_0) \mathcal{I}_{22}^{(2)-1}(\theta_0)\right),
\]
which is a bias-corrected estimtor of the composite likelihood $CL^{(2)}(\alpha_0)$ as shown in Theorem \ref{IC thm}, below. 
That is, $IC^{(2)}$ in Definition \ref{def IC} is the estimated version of $-2|D_n|^{-1}\widetilde{IC}^{(2)}$ since the unknown $\theta_0$ is replaced by $\hat{\theta}_n^{(2)}$. 
The model selection based on $IC^{(2)}$ is to choose the model whose composite likelihood (CL) is maximized, 
which implies that the corresponding model is closest to the truth in the sense of the Kullback-Leibler divergence of the composite likelihoods. 
However, since the maximum CL estimator, $|D_n|^{-1}CL_n^{(2)}(\hat{\alpha}_n^{(2)})$, has an unignorable bias: it can be show that, as $n\to \infty$, 
\[
\mathbb{E}_{\theta_0}\left[|D_n|^{-1}CL_n^{(2)}(\hat{\alpha}_n^{(2)}) - CL^{(2)}(\alpha_0)\right] = O(|D_n|^{-1}). 
\]
we should correct the bias to estimate the value of the composite likelihood.

\begin{thm}\label{IC thm}
Under Assumptions \ref{krho} and \ref{regularity}, it holds that
\begin{eqnarray}\label{IC asymp}
\mathbb{E}_{\theta_0}\left[\widetilde{IC}^{(2)} - CL^{(2)}(\alpha_0)\right] = o(|D_n|^{-1}),\quad n \rightarrow \infty. \label{eq:ic-bias}
\end{eqnarray}
That is, $\widetilde{IC}^{(2)}$ is the asymptotically unbiased estimator of $CL^{(2)}(\alpha_0)$.
\end{thm}

\begin{proof}
Let us evaluate the bias of $CL_n^{(2)}(\alpha_0)$:
\begin{eqnarray*}
Bias 
&=&\mathbb{E}_{\theta_0}\left[|D_n|^{-1}CL_n^{(2)}(\hat{\alpha}_n) - \int_{D_n} \log \left\{\frac{\tilde{\rho}^{(2)}_{\hat{\alpha}_n^{(2)}}(0,u)}{\tilde{K}_{w,2}(r:\hat{\alpha}_n)^{(2)}}\right\}w_r(0,u)\rho_{\theta_0}^{(2)}(0,u)du\right]\\
&=:& d_1 + d_2 + d_3,
\end{eqnarray*}
where 
\begin{eqnarray*}
d_1
&=& |D_n|^{-1}\mathbb{E}_{\theta_0}\left[CL_n^{(2)}(\hat{\alpha}_n^{(2)}) - CL_n^{(2)}(\alpha_0)\right],\\
d_2
&=& \mathbb{E}_{\theta_0}\left[|D_n|^{-1}CL_n^{(2)}(\alpha_0) - \int_{D_n}\log \left\{\frac{\tilde{\rho}^{(2)}_{{\alpha}_0}(0,u)}{\tilde{K}_{w,2}(r:{\alpha}_0)}\right\}w_r(0,u)\rho_{\theta_0}^{(2)}(0,u)du\right]
\end{eqnarray*}
and 
\begin{eqnarray*}
d_3
&=& \mathbb{E}_{\theta_0} \left[ \int_{D_n}\log \left\{\frac{\tilde{\rho}^{(2)}_{{\alpha}_0}(0,u)}{\tilde{K}_{w,2}(r:{\alpha}_0)}\right\}w_r(0,u)\rho_{\theta_0}^{(2)}(0,u)du \right.\\
&&\left. -\int_{D_n}\log \left\{\frac{\rho_{\hat{\alpha}_n^{(2)}}^{(2)}(0,u)}{\tilde{K}_{w,2}(r:\hat{\alpha}_n^{(2)})}\right\}w_r(0,u)\rho_{\hat{\theta}_n^{(2)}}^{(2)}(0,u)du \right].
\end{eqnarray*}
Then the proof ends if we show that 
\[
Bias = \frac{1}{|D_n|}\tr\left(\Sigma_{22}^{(2)}(\theta_0) \mathcal{I}_{22}^{(2)-1}(\theta_0)\right) + o(|D_n|^{-1}).
\]

First, it follows from the definition that $d_2 =0$.

Second, to evaluate $d_3$, we use Taylor expansion for the expected $\log$-composite likelihood and Theorem \ref{pldi DPP} to deduce that
\begin{eqnarray*}
d_3 
&=& \frac{1}{2} \mathbb{E}_{\theta_0}\left[
(\hat{\alpha}_n^{(2)} - \alpha_0)^\top \mathcal{I}_{22}^{(2)}(\theta_0) (\hat{\alpha}_n^{(2)} - \alpha_0)
\right] + O(|D_n|^{-\frac{3}{2}})\\
&=& 
\frac{1}{2}\mathbb{E}_{\theta_0}\left[\tr\left(\mathcal{I}_{22}^{(2)} (\theta_0)
(\hat{\alpha}_n^{(2)} - \alpha_0)  (\hat{\alpha}_n^{(2)} - \alpha_0)^\top \right)
\right] + O(|D_n|^{-\frac{3}{2}})\\
&=& 
\frac{1}{2 |D_n|} \tr\left(\mathcal{I}_{22}(\theta_0) 
\mathbb{E}_{\theta_0} \left[\left\{\sqrt{|D_n|}\left(\hat{\alpha}_n^{(2)} - \alpha_0\right)\right\}^{\otimes 2}\right]\right) + O(|D_n|^{-\frac{3}{2}})\\
&=& \frac{1}{2 |D_n|} \tr\left( \Sigma_{22}^{(2)}(\theta_0) \mathcal{I}_{22}^{(2)-1}(\theta_0)\right)
+ O(|D_n|^{-\frac{3}{2}}).
\end{eqnarray*}
Finally, as well as $d_3$, it holds that 
\[
d_1 = \frac{1}{2 |D_n|} \tr\left( \Sigma_2(\theta_0) \mathcal{I}_2^{-1}(\theta_0)\right)
+ O(|D_n|^{-\frac{3}{2}}).
\]
This competes the proof. 
\qed
\end{proof}
\begin{rem}
Although the estimator for bias is written in the closed form theoretically, the numerical computation of $IC^{(p)}\ (p \geq 2)$ is
not so easy since it has multiple integrals.
\end{rem}
\section{Numerical studies}\label{numerical}
In this section, we will illustrate the finite sample performance of the 
second order composite likelihood estimators for $d$-dimensional 
stationary parametric models of DPPs.
We consider three competing models of stationary DPPs with 
Gaussian kernel $K_\theta^G$, 
the Laplace kernel $K_\theta^L$ and 
the Cauchy kernel $K_\theta^C$ with known shape parameter $\nu$,
where $\theta = (\lambda, \alpha) \in \mathbb{R}^2$,
\[
K_\theta^G(x, y) = \lambda \exp \left(-
\frac{|x-y|^2}{\alpha^2}
\right),\quad x, y \in \mathbb{R}^2,
\]
\[
K_\theta^L(x, y) =
\lambda \exp\left(
- \frac{|x-y|}{\alpha}
\right),\quad x, y \in \mathbb{R}^2.
\]
and 
\[
K_\theta^C(x,y) =
\frac{\lambda}{\left(1 + |x-y|^2/\alpha^2\right)^{\nu + 1}},\quad x, y \in \mathbb{R}^2.
\]
These parametric models are introduced in \cite{lavancier16} and we can simulate samples from them by using the R package ``spatstat''.
The samples are generated in the rectangle
$[0,n] \times [0,n] \subset \mathbb{R}^2$ with $n=5$ and $n=10$ and 
we fix the shape parameter $\nu=1$ for the Cauchy kernel and the true parameter is given by $\lambda_0 = 10.0$ and $\alpha_0 = 0.1$.
Using 500 replications, we calculate the average and unbiased standard deviations of each second order composite likelihood estimators with $r=n/8$ given in Definition \ref{def of estimator}.

Tables 1-3 show the mean and unbiased standard deviation s.d. through
500 replications.
We see that the estimators work well for these parametric models since the standard deviation is very small.
However, $\hat{\alpha}_n^{(2)}$'s for the Laplace and the 
Cauchy DPPs seem to be under estimated.
This fact may indicate that we should consider the bias 
as is discussed in Subsection \ref{IC}.
The estimator of the bias introduced in Definition \ref{def IC} includes 
some multiple integrals which is difficult to compute.
Moreover, to estimate the shape parameter $\nu$ in the 
Cauchy DPP is 
difficult at least by using the second order composite likelihood method.
These kind of parameters may be estimated well by using the higher order composite likelihood method introduced in this paper. 
However, it will be computationally hard since estimating function for higher order estimator includes a multiple integral in $K$-function.
This should be studied more in the future.
\begin{table}[!h]
\begin{center}
\small
\begin{tabular}{|c|c|c|c|}\hline
Case &$n = 5$  &$n = 10$&True\\ 
$\hat{\lambda}_n$&9.93488&9.92392000&10.0\\
						&(0.57305932)&(0.295398636)& \\
$\hat{\alpha}_n^{(2)}$&0.09144804&0.09032759&0.1\\ 
						&(0.01642212)&(0.008114093)& \\\hline
\end{tabular}
\caption{Mean (s.d.) of $\hat{\lambda}_n, \hat{\alpha}_n^{(2)}$ 
thorough 500 replications for Gaussian DPPs.}
\label{Gauss simulation}
\end{center}
\end{table}
\begin{table}[!h]
\begin{center}
\small
\begin{tabular}{|c|c|c|c|}\hline
Case &$n = 5$  &$n = 10$&True\\ \
$\hat{\lambda}_n$&9.92704&9.91238&10.0\\
						&(0.58078383)  &(0.28089632)& \\
$\hat{\alpha}_n^{(2)}$&0.07595449&0.07253522&0.1\\ 
						&(0.02892498)&(0.01275654)& \\\hline
\end{tabular}
\caption{Mean (s.d.) of $\hat{\lambda}_n, \hat{\alpha}_n^{(2)}$ 
thorough 500 replications for the Laplace DPPs.}
\label{Laplace simulation}
\end{center}
\end{table}
\begin{table}[!h]
\begin{center}
\small
\begin{tabular}{|c|c|c|c|}\hline
Case &$n = 5$  &$n = 10$&True\\ 
$\hat{\lambda}_n$&9.98496&9.98212000&10.0\\
						&(0.5655134)  &(0.293259455)& \\
$\hat{\alpha}_n^{(2)}$&0.08814025&0.08522139&0.1\\ 
						&(0.01942937)&(0.009338469)& \\\hline
\end{tabular}
\caption{Mean (s.d.) of $\hat{\lambda}_n, \hat{\alpha}_n^{(2)}$ 
thorough 500 replications for the Cauchy DPPs ($\nu = 0.5$).}
\label{Cauchy simulation}
\end{center}
\end{table}

\end{document}